\documentclass[12pt]{amsart}
\usepackage{latexsym}
\usepackage{a4}
\usepackage{amssymb}
\usepackage{amsbsy}
\newtheorem{thm}{Theorem}[section] 
 
\newtheorem{lem}[thm]{Lemma}

\theoremstyle{definition}

\numberwithin{equation}{section} 

\newcommand{\N}{\Bbb{ N}}

\newcommand{\setsuchthat}{\,\, \pmb{|} \,\,}

\DeclareMathOperator{\Pol}{Pol}
\DeclareMathOperator{\Inv}{Inv}
\newcommand{\PIC}{\Pol \, \Inv C}
\newcommand{\Loc}[1]{\Pol \, \Inv^{[#1]}}
\DeclareMathOperator{\Loo}{Loc}
\newcommand{\Lo}[1]{\Loo_{#1}}
\newcommand{\ar}[1]{^{[#1]}}

\newcommand{\algop}[2]{\langle {#1}, {#2} \rangle}

\title{On the local closure of clones on countable sets}
\author{Erhard Aichinger}
\thanks{Supported by the Austrian Science Fund (FWF): P24077}
\begin{document}
\begin{abstract}
  We consider clones on countable sets. If such a clone
  has quasigroup operations, is locally closed and countable,
  then there is a function $f : \N \to \N$ such that
  the $n$-ary part of $C$ is equal to the $n$-ary part
  of $\Loc{f(n)} C$, where $\Inv\ar{f(n)} C$ denotes the set of
  $f(n)$-ary invariant relations of $C$.
\end{abstract}
\maketitle

\section{Results} \label{sec:res}
We investigate clones on infinite sets \cite{PK:FUR, Sz:CIUA, GP:ASOC}.
For a clone $C$ on
$A$, its local closure $\overline{C}$ consists of all those finitary operations
on $A$ that can be interpolated at each finite subset of their domain
by a function in $C$, and we have $\overline{C} = \PIC$. Here, as in \cite{PK:FUR},
$\Inv C$ denotes the set of those finitary relations on $A$ that are preserved
by all functions in $C$, and for a set $R$  of relations on $A$, $\Pol \, R$ denotes the
set of those finitary operations on $A$ that preserve all relations in $R$.
A clone is called \emph{locally closed} if it is equal to
its local closure. 
$C$ is called a \emph{clone with quasigroup operations}
if there are three binary operations $\cdot, \backslash, / \, \in C$ such that
$\algop{A}{\cdot, \backslash, /}$ is a quasigroup \cite[p.24]{BS:ACIU}.
Theorem~\ref{thm:1} states that a 
clone with quasigroup operations on a countable set
is either locally closed, or its local closure  $\PIC$ is uncountable. 
\begin{thm} \label{thm:1}
   Let $A$ be a set with $|A| = \aleph_0$, and let $C$ be a clone
   with quasigroup operations. If $\left|\PIC\right| \le \aleph_0$, then
   $C = \PIC$.
\end{thm}
This theorem does not hold for clones without quasigroup operations.
We say that $C$ is \emph{constantive} if it contains
all unary constant operations. 
\begin{thm} \label{thm:2}
   There exist a set $A$ with $|A| = \aleph_0$ and a constantive clone
   $C$ on $A$ such that $\left|\PIC\right| = \aleph_0$ and $C \neq \PIC$.
\end{thm}
For a clone $C$ on $A$, 
 $\Inv\ar{m} C$ denotes the set of $m$-ary invariant relations of $C$. 
It is well known that a function $f$ lies in $\Loc{m} C$ if and only if
it can be interpolated at every $m$-element subset of its domain by
a function in $C$; this is discussed, e.g.,  in \cite{Po:AGGT} and in \cite[Lemma~7]{Ei:EBUC} and stated in Lemma~\ref{lem:locint}.
We write $C\ar{n}$ for the set of $n$-ary functions in $C$.
Let $B$ be any set, and let $F \subseteq A^B$. A subset $D$ of $B$ is
    a \emph{base of equality} for $F$ if for all $f,g \in F$ with
    $f|_D = g|_D$, we have $f=g$.
    Theorem~\ref{thm:1} can be extended in the following way:
\begin{thm} \label{thm:nc}
    Let $A$ be a set with $|A| = \aleph_0$, and let $C$ be a 
    clone
   on $A$ with quasigroup operations.
   Then the following are equivalent:
   \begin{enumerate}
    \item \label{it:m1} $\left|\PIC\right| \le \aleph_0$.
    \item \label{it:m3} For each $n \in \N$, 
     $C\ar{n}$ has a finite base of equality.
     \item \label{it:m5} $|C| \le \aleph_0$ and $\forall n \in \N$ $\exists k \in \N$ : $C\ar{n} = (\Loc{k} C)\ar{n}$.
    \item  \label{it:m6} $|C| \le \aleph_0$ and $C = \PIC$.
   \end{enumerate}
\end{thm}
A weaker version of this result was proved in \cite{Ai:LPFO}. As an application,
we obtain, e.g.,  that a countably infinite integral domain $R$ cannot be affine complete:
If it is affine complete, then the clone $C$ of polynomial functions of $R$ satisfies~\eqref{it:m5}, 
and therefore the unary polynomials
have a finite base of equality $D$. But $f(x) = 0$ and $g(x) = \prod_{d \in D} (x - d)$
show that this is not possible. In fact, Theorem~\ref{thm:nc} extracts a common
idea of several ``non-affine completeness'' results \cite{Ha:CoNa, KK:AfCG}. The proofs
are given in Section~\ref{sec:proofs}.

\section{Finite bases of equality}
Theorems~\ref{thm:1} and~\ref{thm:nc} rely on the following observation.
In a less general context, this observation appears in
\cite[Theorem~2]{Ai:LPFO}, and large parts of its 
proof are verbatim copies from \cite{Ai:LPFO}
and \cite[pp.51-52]{Ai:TSOC}.
\begin{lem} \label{lem:BE} 
    Let $A$ be a set with $|A| = \aleph_0$, let $m \in \N$, and let $C$ be a clone on $A$
    with quasigroup operations. If $|(\PIC)^{[m]}| \le \aleph_0$,
    then $C^{[m]}$ has a finite base of equality.
\end{lem}

\begin{proof}%
       Let $\overline{C} := \PIC$. In the case that  $\overline{C}^{[m]}$ is finite,
       its subset $C^{[m]}$ is also finite. Then
       for every $f, g \in C^{[m]}$ with $f \neq g$, we choose
       $a_{(f,g)} \in A^m$ such that $f (a_{(f,g)}) \neq g (a_{(f,g)})$.
       Then $D := \{ a_{(f,g)} \setsuchthat f,g \in C^{[m]}, f \neq g \}$ is a base of
       equality for $C^{[m]}$.
       Hence we will from now on assume $|\overline{C}^{[m]}| = \aleph_0$.
       Let $a_0, a_1, a_2, \ldots$ 
       and $f_0, f_1, f_2, \ldots$ 
       be complete enumerations of $A^m$ and $\overline{C}^{[m]}$, respectively.
       Furthermore we abbreviate 
       the set $\{ a_i \,|\, i \le r \}$ by $A (r)$.
       Seeking a contradiction, we suppose that there is no finite 
       base of equality for $C^{[m]}$. 
We shall construct a sequence $(n_k)_{k \in \N_0}$ of 
non-negative integers 
and a sequence $(g_k)_{k \in \N_0}$ of elements of $C^{[m]}$ with the following 
properties: 
\begin{enumerate} 
   \item $ \forall k \in \N_0 :
           g_k|_{A (n_k)}
           \not= 
           f_k|_{A (n_k)}$, 
   \item $ \forall k \in \N_0  :  n_{k + 1} > n_{k}$ 
   \item $ \forall k \in \N_0  :  
           g_{k+1} |_{ A ({n_k}) } = 
           g_{k  } |_{ A ({n_k}) } $. 
\end{enumerate} 
We construct the sequences inductively. We choose $g_0 \in C^{[m]}$ such 
that $g_0 \not= f_0$, and $n_0 \in \N_0$ minimal with 
$g_0 (a_{n_0}) \not= f_0 (a_{n_0})$.
If we have already constructed $g_k$ and $n_k$ we construct 
$g_{k+1}$ and $n_{k+1}$ as follows: 
in the case that $g_k|_{ A (n_k) }  \not= 
              f_{k+1}|_{ A (n_k) } $, 
we set $g_{k+1} := g_k$ and 
$n_{k+1} := n_k + 1$. 
In the case $g_k|_{ A (n_k) }  =  
              f_{k+1}|_{ A (n_k) } $, we first show
that  
there exists a function $h \in C^{[m]}$ with 
\begin{equation} \label{eq:hf} 
      g_k |_{A (n_k) } = h|_{A (n_k) }  \text{ and } h \not= f_{k+1}.
\end{equation}
  Suppose that on the contrary every $h \in C^{[m]}$ with 
 $g_k |_{A (n_k) } = 
    h|_{A (n_k) }$ satisfies $h = f_{k+1}$.
 In this case, $g_k = f_{k+1}$, and therefore $f_{k+1} \in C^{[m]}$.
   We will show next that $A(n_k)$ is a base of equality of $C^{[m]}$.
 To this end, let $r,s \in C^{[m]}$ with $r|_{A (n_k)} = s|_{A (n_k)}$. 
 We define $t(x) := r(x) \backslash (s(x) \cdot f_{k+1} (x))$.
 Then for every $x \in  A (n_k)$, we have $t(x) = r(x) \backslash (r(x) \cdot
 f_{k+1} (x)) = f_{k+1} (x) = g_k (x)$. Hence $t = f_{k+1}$.
 Therefore, for every $x \in A^m$, we have
 $r(x) \backslash (s (x) \cdot f_{k+1} (x)) = f_{k+1} (x)$, thus
 $s(x) \cdot f_{k+1} (x) = r(x) \cdot f_{k+1} (x)$, and therefore
 $(s (x) \cdot f_{k+1} (x)) / f_{k+1} (x) = (r (x) \cdot f_{k+1} (x)) / f_{k+1} (x))$,
 which implies $s(x) = r(x)$. Thus $r=s$, which completes the proof
 that $A (n_k)$ is a base of equality of $C^{[m]}$, contradicting the
 assumption that no such base exists. Hence there is $h \in C\ar{m}$ that
 satisfies~\eqref{eq:hf}.
 Continuing in the
 construction of $g_{k+1}$, we set $g_{k+1} := h$, and we choose $n_{k+1}$ to be minimal with 
    $h (a_{n_{k+1}}) \not= f_{k+1} (a_{n_{k+1}})$. 
  
Since for every $a \in A^m$, the sequence  
$(g_k (a))_{k \in \N_0}$ 
is eventually constant, we may define a function 
$l :A^m \to A$ by 
\(
    l(a) := \lim_{k \rightarrow \infty} g_k (a).
\)
We will now show that $l \in \overline{C}^{[m]}$. The clone $\overline{C}$ contains exactly
those functions that can be interpolated at every finite subset of their
domain with a function in $C$. Hence we show that $l$ can be interpolated
at every finite subset $B$ of $A^m$ by a function in $C$.
Since $\bigcup_{i \in \N_0} A_i = A^m$, there is $k \in \N$  such that
$B \subseteq A (n_k)$. Since $l|_{A (n_k)} = g_k|_{A (n_k)}$, the function $g_k \in C^{[m]}$
interpolates $l$ at $B$.
We conclude that that the
function  $l$ lies in  $\overline{C}^{[m]}$. Thus $l$ is
equal to $f_k$ for 
some $k \in \N_0$. Since $ l|_{ A (n_k) } = 
                          g_k|_{ A (n_k)  }$   
and $g_k|_{ A (n_k) } \not= 
     f_k|_{ A (n_k) }$, 
     we obtain  $l|_{A (n_k)} \not= f_k|_{A (n_k)}$, a contradiction.
Hence $C^{[m]}$ has a finite base of equality. 
\end{proof}

\begin{lem}[cf. {\cite[Lemma~1]{HN:LPFO} and \cite[Proposition~2]{Ai:LPFO}}]  \label{lem:boe}
   Let $A$ be a set, let $C$ be a clone on $A$, let $n \in \N$,  let $D$ be a finite base
   of equality for $C\ar{n}$, and let $k := |D| + 1$.
    Then $C\ar{n} = (\Loc{k} C)\ar{n}$.
\end{lem}
\begin{proof}
 Let $l \in (\Loc{k} C)\ar{n}$. Then $l$ can be interpolated at
 every subset of $A^n$  with at most $k$ elements by a function in $C \ar{n}$.
 Hence there is $f \in C^{[n]}$ such that $f|_{D} = l|_{D}$. If $f = l$,
then $l \in C^{[n]}$. In the case $f \neq l$, we take $y \in A^n$
such that $f(y) \neq l(y)$. Now we choose $g \in C^{[n]}$ such
that $g|_{D \cup \{y\}} = l|_{D \cup \{y\}}$. Then $f(y) \neq g(y)$
and $f|_{D} = g|_{D}$,  contradicting the assumption that $D$ is a base of equality for
  $C\ar{n}$. 
\end{proof}

\section{A compactness property for local interpolation}

For two sets $A$ and $B$, a set of functions $F \subseteq A^B$, and $k \in \N$,
the set $\Lo{k} F$ is defined as the set of those functions that can be interpolated
at every subset of $B$ with at most $k$ elements by a function in $F$ \cite{Po:AGGT}.
If $C$ is a clone, and $F = C^{[m]}$ is its $m$-ary part, then $\Lo{k} (C^{[m]})$
is the set of  $m$-ary functions on $A$  that preserve the $k$-ary relations in $\Inv \, C$.
\begin{lem} (cf. \cite[p.\ 31, Theorem~4.1]{Po:AGGT}) \label{lem:locint}
  Let $A$ be a set, let $C$ be a clone on $A$, and let $k, m \in \N$.
  Then
  $\Lo{k} (C\ar{m}) = (\Loc{k} C)\ar{m} = (\Loc{k} (C\ar{m}))\ar{m}$.
\end{lem}
For countable sets $A$, we obtain the following result.
\begin{thm} \label{thm:tcomp}
  Let $A$ be a set with $|A| \le \aleph_0$, and let $C$ be a clone on $A$ with quasigroup operations
  such that $|C\ar{m}| \le \aleph_0$.
  If $\bigcap_{k \in \N} \Lo{k} (C\ar{m}) = C\ar{m}$, then there exists $n \in \N$
  such that $\Lo{n} (C\ar{m}) = C\ar{m}$.
\end{thm}  
\emph{Proof:}
By Lemma~\ref{lem:locint} and the assumptions, $C\ar{m} =
                            \bigcap_{k \in \N} \Lo{k} (C\ar{m}) =
                            \bigcap_{k \in \N} (\Loc{k} C)\ar{m} =  (\PIC)\ar{m}$. Now
                            Lemma~\ref{lem:BE} yields a finite base of equality for $C\ar{m}$,
                            and now by Lemma~\ref{lem:boe}, there is $n \in \N$ such that
                            $C\ar{m} = (\Loc{n} C)\ar{m} = \Lo{n} (C\ar{m})$. \qed

For an arbitrary $m$-ary operation $f$ on the set $A$, we say that the property $I(f,n,C)$ holds
if $f$ can be interpolated by a function in $C$ at each subset of $A^m$ with at most $n$
elements. Theorem~\ref{thm:tcomp} yields the following compactness property: if $C$ is a countable clone with quasigroup operations,
if $A$ is countable, and if
$\forall f \in A^{A^m} : ((\forall k \in \N: I(f,k,C)) \Rightarrow f \in C)$ holds, then there is a natural number $n \in \N$
such that
$\forall f \in A^{A^m} : (I(f,n,C) \Rightarrow f \in C)$ holds.

\section{Proofs of the Theorems from Section~\ref{sec:res}} \label{sec:proofs}    

\begin{proof}[Proof of Theorem~\ref{thm:nc}:]

    \eqref{it:m1}$\Rightarrow$\eqref{it:m3}: Let $n \in \N$. Since $(\PIC)\ar{n} \subseteq \PIC$, we have
       $|(\PIC)\ar{n}| \le \aleph_0$.  Lemma~\ref{lem:BE} now yields a finite base of equality
      for $C\ar{n}$.

      \eqref{it:m3}$\Rightarrow$\eqref{it:m5}: Let $n \in \N$, and 
      let $D \subseteq A^n$ be a finite base of equality for $C\ar{n}$. We set $k := |D|+1$
      and  obtain $C\ar{n} = (\Loc{k} C)\ar{n}$ from Lemma~\ref{lem:boe}.
      The mapping $\varphi : C\ar{n} \to A^D$, $f \mapsto f|_D$ is injective, therefore
      $|C\ar{n}| \le \aleph_0$. 
       Since for every $n \in \N$, we have $|C\ar{n}| \le \aleph_0$, we have $|C| \le \aleph_0$.

     \eqref{it:m5}$\Rightarrow$\eqref{it:m6}: Let $n \in \N$,
       and let $k$ be taken from~\eqref{it:m5}.
     Then 
         $(\PIC)\ar{n} \subseteq (\Loc{k} C) \ar{n} \subseteq C\ar{n}$. 

     \eqref{it:m6}$\Rightarrow$\eqref{it:m1}: Obvious.
\end{proof}

\begin{proof}[Proof of Theorem~\ref{thm:1}:]
 Theorem~\ref{thm:1} is the implication~\eqref{it:m1}$\Rightarrow$\eqref{it:m6}
 of Theorem~\ref{thm:nc}.
\end{proof}

\begin{proof}[Proof of Theorem~\ref{thm:2}:]
   Let $A := \N_0$, and let $p (x) := x \, \text{ mod } 2$ for all $x \in \N_0$.
   For $a \in \N_0$, we define $g_a : \N_0 \to \N_0$ by
   \[
       g_a (x) := \left\{ \begin{array}{l}
                            p(x) \text{ if } x < a, \\ 
                            x \text{ if } x \ge a,
                          \end{array}
                 \right.
   \]
   and we let $c_a (x) := a$ for all $x \in \N_0$. 
   Let $M := \{ g_a \setsuchthat a \in \N_0 \} \cup \{ c_a (x) \setsuchthat a \in \N_0 \}$.
   We will first show that 
   $\algop{M}{\circ, g_0}$ is a submonoid of $\algop{{\N_0}^{\N_0}}{\circ, \mathrm{id}_{\N_0}}$.
   To this end, it is suffcient to show that $g_a \circ g_b \in M$ for all $a, b \in \N_0$.
   Since $g_0 = g_1 = g_2 = \mathrm{id}_{\N_0}$, we may assume $a \ge 3$ and $b \ge 3$.
   We will show
   \begin{equation} \label{eq:gab}
            g_a (g_b (x)) := g_{\max (a,b)} (x) \text{ for all } x \in \N_0.
   \end{equation}
   In the case $x < b$, we have $g_a (g_b (x)) = g_a (p(x)) = p(p(x)) = p(x) = g_{\max(a,b)} (x)$.
   In the case that $x \ge b$ and $x < a$, we have $g_a (g_b (x)) = g_a (x)$, and since in this
   case $b \le a$, $g_a(x) = g_{\max (a,b)} (x)$.
   In the case that $x \ge a$ and $x \ge b$, we have $g_a (g_b(x)) = g_a (x) = x = g_{\max (a,b)} (x)$.
   From~\eqref{eq:gab}, we deduce that $M$ is closed under composition.
   Now let $C$ be the clone on $\N_0$ that is generated by $M$; this clone consists
   of all functions $(x_1,\ldots, x_n) \mapsto m (x_j)$ with $n, j \in \N$, $m \in M$ and
   $j \le n$.
   Let $\overline{C} := \PIC$. Next, we show 
   \begin{equation} \label{eq:p}
          p \in \overline{C}.
   \end{equation}
   To prove~\eqref{eq:p}, we show that $p$ can be interpolated at every finite subset
   $B$ of $\N_0$ by a function in $C$. Let $a := \max (B)$. Then
   $g_{a+1} |_B = p|_B$. This completes the proof of~\eqref{eq:p}.
   Now we show
   \begin{equation} \label{eq:cp}
        \overline{C}^{[1]} = C^{[1]} \cup \{ p \}.
   \end{equation} 
   We only have to establish $\subseteq$.
   It is helpful to write down the list of values of some of the functions in $M \cup \{p\}$.
   \[
       \begin{array}{rl}
             c_3 & 333333\ldots \\
             c_2 & 222222\ldots \\
             c_1 & 111111\ldots \\
             c_0 & 000000\ldots \\
             p   & 010101\ldots \\
             \mathrm{id} & 012345\ldots \\
             g_3 & 010345\ldots \\
             g_4 & 010145\ldots \\
             g_5 & 010105\ldots
       \end{array}
   \]  
   Let $f \in \overline{C}^{[1]}$ with $f \neq p$, and let 
   $k \in N_0$ be minimal with $f(k) \neq p(k)$.
   Let $g \in C^{[1]}$ be such that $g|_{\{0,\ldots, k\}}
   = f|_{\{0,\ldots, k\}}$.
   We distinguish three cases.
   \begin{itemize}
     \item Case $k=0$: Then $g(0) \neq 0$, and therefore
   $g = c_{g(0)}$. If $f = c_{g(0)}$, we have $f \in C$. If $f \neq c_{g(0)}$, we
   let $y$ be minimal with $f(y) \neq g(0)$. We interpolate $f$ 
   at $\{0, y \}$ by a function $h \in C$. This function  $h$ is not
   constant and satisfies
   $h(0) \neq 0$. Such a function does not exist in $C$, therefore
    the case $f \neq c_{g(0)}$ cannot occur.
     \item Case $k = 1$: Then $g(1) \neq 1$. By examining the functions in $M$,
   we see that $g = c_0$. If $f = c_0$, we have $f \in C$. If $f \neq c_{0}$, we
    let $y$ be minimal with $f(y) \neq 0$. Interpolating $f$
   at $\{0,1, y\}$ by $h \in C$, we obtain a function $h \in C$
   with $h(0) = h(1) = 0$ and $h (y) \neq 0$. Such a function
   does not exist in $C$; this  contradiction shows $f = c_0$ and
   therefore $f \in C$.
    \item Case $k \ge 2$: Then $g = g_k$. If $f = g_k$, then
   $f \in C$. If $f \neq g_k$, we choose $y$ minimal with
   $f(y) \neq g_k (y)$ and interpolate $f$ at $\{0,1,\ldots, k\} \cup \{y\}$
   by a function $h \in C$. Again, such a function is not
   available in $C$, and therefore $f = g_k \in C$.
  \end{itemize}
  Thus every $f \in \overline{C}^{[1]}$ with $f\neq p$ is an element
  of $C$.
  By its definition, $C$ contains all constant unary operations in $\N_0$.
  Since $C$ preserves the relation $\rho = \{ (a,b,c,d) \in A^4 \setsuchthat a = b \text{ or } c = d\}$,
  also $\overline{C}$ preserves $\rho$. Hence by \cite[Lemma~1.3.1(a)]{PK:FUR},
  every function in $\overline{C}$ is essentially unary and hence 
  of the form
  $l(x_1,\ldots, x_n) = f(x_j)$ with $n \in \N$, $j \in \{1,\ldots, n\}$, and
  $f \in \overline{C}\ar{1} = M \cup \{p\}$.
  This implies that $\overline{C}$ is countable.
  The function $p$ witnesses
  $C \neq \overline{C}$. \end{proof}

\section{Constantive Clones}
  In constantive clones, a finite base of equality for the functions of arity $m$
  yields
  finite bases of equality for all other arities. This will allow to refine
  Theorem~\ref{thm:nc}. 
  \begin{lem} \label{lem:3}
      Let $C$ be a clone on the set $A$, let $m \in \N$, and let $D \subseteq A^m$
      be a base of equality for $C\ar{m}$. Then the projection of $D$ to the
      first component $\pi_1 (D)$ is a base of equality
      for $C\ar{1}$.
  \end{lem}
   \begin{proof}
      Let $f,g \in C\ar{1}$ such that $f|_{\pi_1 (D)} = g|_{\pi_1 (D)}$.
                 Let $f_1 (x_1, \ldots, x_m) := f(x_1)$ and 
                     $g_1 (x_1, \ldots, x_m) := g(x_1)$. Then for every
      $(d_1, \ldots, d_m) \in D$, we have
      $f_1 (d_1,\ldots, d_m) = f(d_1) = g(d_1) = g_1 (d_1,\ldots, d_m)$,
      and therefore $f_1 = g_1$, which implies $f=g$. 
   \end{proof}
\begin{lem} \label{lem:2}
     Let $A$ be a set, let $C$ be a constantive clone on $A$, and let $D \subseteq A$ be
     a base of equality for $C^{[1]}$. Then for every $n \in \N$,
     $D^n$ is a base of equality for $C^{[n]}$.
\end{lem}
\begin{proof}
  We proceed by induction on $n$. If $n=1$, $D^1 = D$ is a base of equality
  of $C^{[1]}$ by assumption. For the induction step, let $n \ge 2$, and suppose
  that $D^{n-1}$ is a base of equality for $C^{[n-1]}$.
  Let $f, g \in C^{[n]}$ and assume $f|_{D^n} = g|_{D^n}$.
  We first show
  \begin{equation} \label{eq:1}
       f|_{A \times D^{n-1}} = g|_{A \times D^{n-1}}.
  \end{equation}
  Let $(a, d_2, \ldots, d_n) \in A \times D^{n-1}$, and define
  $f_1 (x) := f(x, d_2, \ldots, d_n)$ and $g_1 (x) := g (x, d_2, \ldots, d_n)$ for $x \in A$.
  Then $f_1, g_1 \in C^{[1]}$ and $f_1|_D = g_1|_D$. Hence $f_1 = g_1$, and thus
  $f (a, d_2, \ldots, d_n) = f_1 (a) = g_1 (a) = g (a, d_2, \ldots, d_n)$, which completes the
  proof of~\eqref{eq:1}.
  We will now prove that $f = g$. Let $(b_1,\ldots, b_n) \in A^n$, and
  define  $f_2 (x_2,\ldots, x_n) := f(b_1, x_2,\ldots, x_n)$, $g_2 (x_2, \ldots, x_n) :=
  g (b_1, x_2, \ldots, x_n)$ for all $x_2, \ldots, x_n \in A$. By~\eqref{eq:1},
  $f_2|_{D^{n-1}} = g_2|_{D^{n-1}}$, and therefore by the induction hypothesis
  $f_2 = g_2$. Thus $f(b_1,\ldots, b_n) = g(b_1, \ldots, b_n)$. 
\end{proof}
  Hence, for constantive clones we can give the following slight refinement of Theorem~\ref{thm:nc}.
\begin{thm} \label{thm:3}
 Let $A$ be a set with $|A| = \aleph_0$, let $C$ be a constantive clone
   on $A$ with quasigroup operations, and let $m \in \N$.
   Then the following are equivalent:
   \begin{enumerate}
    \item \label{it:c1} $|(\PIC)\ar{1}| \le \aleph_0$.
    \item \label{it:c2} $C\ar{1}$ has a finite base of equality.
    \item \label{it:c3} $C\ar{m}$ has a finite base of equality.
    \item \label{it:c4} $|C| \le \aleph_0$ and $\exists d \in \N$ $\forall n \in \N$ : $C\ar{n} = (\Loc{d^n+1} C)\ar{n}$.
    \item \label{it:c5} $|C| \le \aleph_0$ and $\forall n \in \N$ $\exists k \in \N$ : $C\ar{n} = (\Loc{k} C)\ar{n}$.
    \item \label{it:c6} $|C| \le \aleph_0$ and $C = \PIC$.
   \end{enumerate}
\end{thm}
 \begin{proof}
    \eqref{it:c1}$\Rightarrow$\eqref{it:c2}: Lemma~\ref{lem:BE}.

    \eqref{it:c2}$\Rightarrow$\eqref{it:c3}: Lemma~\ref{lem:2}.

    \eqref{it:c3}$\Rightarrow$\eqref{it:c2}: Lemma~\ref{lem:3}.

    \eqref{it:c2}$\Rightarrow$\eqref{it:c4}: Let $D$ be a finite base of equality for $C\ar{1}$. 
           Let $n \in \N$, and set $k := |D|^n + 1$.
         By Lemma~\ref{lem:2}, $D^n$ is a base of equality for $C\ar{n}$, and 
     Lemma~\ref{lem:boe} yields $C\ar{n} = (\Loc{k} C)\ar{n}$. Since $D^n$ is
     a finite base of equality, the mapping $f \mapsto f|_{D^n}$ is an injective
     mapping from $C\ar{n}$ to $A^{D^n}$, making $C\ar{n}$ countable.
     Since $C\ar{n}$ is countable for every $n \in \N$, we obtain $|C| \le \aleph_0$.

    \eqref{it:c4}$\Rightarrow$\eqref{it:c5}: Set $k: = d^n + 1$.

    \eqref{it:c5}$\Rightarrow$\eqref{it:c6}: Let $n \in \N$, and $k$ be produced
      by \eqref{it:c5}. Then
        $(\PIC)\ar{n} \subseteq (\Loc{k} C)\ar{n} = C\ar{n}$.
    
    \eqref{it:c6}$\Rightarrow$\eqref{it:c1}: We have $(\PIC)\ar{1} \subseteq \PIC \subseteq C$.  

     \end{proof}

\section*{Acknowledgements}
The author thanks Mike Behrisch for furnishing information on
the reference \cite{Po:AGGT}.

\bibliographystyle{plain}
\def\cprime{$'$}

\begin{flushleft}
\begin{small}
Institute for Algebra, Johannes Kepler University Linz,
Austria \\
{\tt erhard@algebra.uni-linz.ac.at}\\
\end{small}
\end{flushleft}

\end{document}